\documentclass[11pt]{article}
\usepackage[letterpaper,twoside,outer=1.3in,vmargin=1.3in,]{geometry}
\usepackage{amsmath,amssymb,amsthm,amsfonts,enumerate,times,epsfig,color,hyperref,cleveref}

\newcommand{\ignore}[1]{}

\DeclareMathOperator{\cone}{cone}

\DeclareMathOperator{\aff}{aff}

\DeclareMathOperator{\Diag}{Diag}

\newcommand{\rank}[1]{\mathrm{rank}\!\left(#1\right)}

\newcommand{\1}{\mathbf{1}}
\newcommand{\0}{\mathbf{0}}

\newcommand{\rvline}{\hspace*{-\arraycolsep}\vline\hspace*{-\arraycolsep}}

\newtheorem{theorem}{Theorem}[section]
\newtheorem{CO}[theorem]{Corollary}
\newtheorem{LE}[theorem]{Lemma}
\newtheorem{PR}[theorem]{Proposition}

\newtheorem{CN}[theorem]{Conjecture}
\newtheorem{RE}[theorem]{Remark}

\newcounter{claim_nb}[theorem]
\newtheorem*{claim*}{Claim}

\newcounter{claim_nbs}[section]
\newcounter{subclaim_nb}[claim_nbs]
\title{Total dual dyadicness and dyadic generating sets}

\author{Ahmad Abdi \and
G\'erard Cornu\'ejols \and
Bertrand Guenin \and
Levent Tun\c{c}el
}
\begin{document}

\maketitle

\begin{abstract}
A vector is \emph{dyadic} if each of its entries is a dyadic rational number, i.e. of the form $\frac{a}{2^k}$ for some integers $a,k$ with $k\geq 0$. A linear system $Ax\leq b$ with integral data is \emph{totally dual dyadic} if whenever $\min\{b^\top y:A^\top y=w,y\geq \0\}$ for $w$ integral, has an optimal solution, it has a dyadic optimal solution. In this paper, we study total dual dyadicness, and give a co-NP characterization of it in terms of \emph{dyadic generating sets for cones and subspaces}, the former being the dyadic analogue of \emph{Hilbert bases}, and the latter a polynomial-time recognizable relaxation of the former. Along the way, we see some surprising turn of events when compared to total dual integrality, primarily led by the \emph{density} of the dyadic rationals. Our study ultimately leads to a better understanding of total dual integrality and polyhedral integrality. We see examples from dyadic matrices, $T$-joins, cycles, and perfect matchings of a graph. \end{abstract}

\section{Introduction}

A \emph{dyadic rational} is a number of the form $\frac{a}{2^k}$ for some integers $a,k$ where $k\geq 0$. The dyadic rationals are precisely the rational numbers with a finite binary representation, and are therefore relevant for (binary) floating-point arithmetic in numerical computations. Modern computers represent the rational numbers by fixed-size floating points, inevitably leading to error terms, which are compounded if serial arithmetic operations are performed such as in the case of mixed-integer linear, semidefinite, and more generally convex optimization. This has led to an effort to mitigate floating-point errors~\cite{Wei06} as well as the need for exact solvers~\cite{Cook11,Steffy11}. 

We address a different, though natural theoretical question: \emph{When does a linear program admit an optimal solution whose entries are dyadic rationals?} A vector is \emph{dyadic} if every entry is a dyadic rational. Consider the following primal dual pair of linear programs for $A\in \mathbb{Z}^{m\times n},b\in \mathbb{Z}^m$ and $w\in \mathbb{Z}^n$. $$
(P) \quad\max\{w^\top x:Ax\leq b\} \qquad (D) \quad\min\{b^\top y:A^\top y = w, y\geq \0\}.
$$ (${\bf 0}$ and ${\bf 1}$ denote respectively the all-zeros and all-ones column, or row, vectors of appropriate dimension.) When does (D) admit a dyadic optimal solution for all $w\in \mathbb{Z}^n$? How about (P)? Keeping close to the integral case, these questions lead to the notions of \emph{totally dual dyadic} systems and \emph{dyadic polyhedra}. In this paper, we reassure the reader that dyadic polyhedra enjoy a similar characterization as \emph{integral polyhedra}, but in studying totally dual dyadic systems, we see an intriguing and somewhat surprising turn of events when compared to \emph{totally dual integral (TDI)} systems~\cite{Edmonds77}. As such, we shall keep the focus of the paper on total dual dyadicness and its various characterizations. The characterizations lead to \emph{dyadic generating sets} for cones and subspaces, where the first notion is polyhedral and can be thought of as a dyadic analogue of \emph{Hilbert bases}, while the second notion is lattice-theoretic and new. We shall see some intriguing examples of totally dual dyadic systems and dyadic generating sets from Integer Programming, Combinatorial Optimization, and Graph Theory. Our study eventually leads to a better understanding of TDI systems and integral polyhedra.

Our characterizations extend easily to the \emph{$p$-adic rationals} for any prime number $p\geq 3$. For this reason, we shall prove our characterizations in the general setting. Interestingly, however, most of our examples do \emph{not} extend to the $p$-adic setting for $p\geq 3$.

\subsection{Totally dual $p$-adic systems and $p$-adic generating sets}

Let $p\geq 2$ be a prime number. A \emph{$p$-adic rational} is a number of the form $\frac{a}{p^k}$ for some integers $a,k$ where $k\geq 0$. A vector is \emph{$p$-adic} if every entry is a $p$-adic rational. Consider a linear system $Ax\leq b$ where $A\in \mathbb{Z}^{m\times n},b\in \mathbb{Z}^m$. We say that $Ax\leq b$ is \emph{totally dual $p$-adic} if for all $w\in \mathbb{Z}^n$ for which $\min\{b^\top y:A^\top y = w, y\geq \0\}$ has an optimum, it has a $p$-adic optimal solution. For $p=2$, we abbreviate `totally dual dyadic' as `TDD'. We prove the following characterization, which relies on two key notions defined afterwards.

\begin{theorem}[proved in \S\ref{sec:TDp}]\label{TDp-char} 
Let $A\in \mathbb{Z}^{m\times n},b\in \mathbb{Z}^m$ and
$P:=\{x:Ax\leq b\}$. Given a nonempty face $F$, denote by $A_Fx\leq b_F$ the subsystem of $Ax\leq b$ corresponding to the implicit equalities of $F$. Then the following statements are equivalent for every prime $p$: (1) $Ax\leq b$ is totally dual $p$-adic,
(2) for every nonempty face $F$ of $P$, the rows of $A_F$ form a $p$-adic generating set for a cone,
(3) for every nonempty face $F$ of $P$, the rows of $A_F$ form a $p$-adic generating set for a subspace.
\end{theorem}

In fact, in (2), it suffices to consider only the minimal nonempty faces.

Let $\{a^1,\ldots,a^n\}\subseteq \mathbb{Z}^m$. The set $\{a^1,\ldots,a^n\}$ is a \emph{$p$-adic generating set for a cone ($p$-GSC)} if every integral vector in the conic hull of the vectors can be expressed as a $p$-adic conic combination of the vectors (meaning that the coefficients used are $p$-adic). In contrast, $\{a^1,\ldots,a^n\}$ is a \emph{$p$-adic generating set for a subspace ($p$-GSS)} if every integral vector in the linear hull of the vectors can be expressed as a $p$-adic linear combination of the vectors. For $p=2$, we use the acronyms DGSC and DGSS instead of $2$-GSC and $2$-GSS, respectively. 

The careful reader may notice that an \emph{integral} generating set for a cone is just a \emph{Hilbert basis}~\cite{Giles79}.\footnote{Some sources, including \cite{Schrijver03}, \S5.17, define Hilbert bases in a different way, and the reader should bear this in mind when studying examples. Throughout the paper, we follow the convention implied in the original paper~\cite{Giles79}, and explicitly stated in \cite{Schrijver98}, \S22.3.} In a departure from Hilbert bases, where a satisfying characterization remains elusive, we have the following polyhedral characterization of a $p$-GSC:

\begin{theorem}[proved in \S\ref{sec:generating}]\label{pGSC-char}
Let $\{a^1,\ldots,a^n\}\subseteq \mathbb{Z}^m$, $C:=\cone\{a^1,\ldots,a^n\}$, and $p$ a prime. Then $\{a^1,\ldots,a^n\}$ is a $p$-GSC if, and only if, for every nonempty face $F$ of $C$, $\{a^i:a^i\in F\}$ is a $p$-GSS.
\end{theorem}

The careful reader may notice that in contrast to total dual integrality, the characterization of totally dual $p$-adic systems, \Cref{TDp-char}, enjoys a third equivalent condition, namely (3). This new condition, as well as the characterization of a $p$-GSC, \Cref{pGSC-char}, is made possible due to a distinguishing feature of the $p$-adic rationals: \emph{density}. The $p$-adic rationals, as opposed to the integers, form a dense subset of~$\mathbb{R}$. We shall elaborate on this in \S\ref{sec:density}.

Going further, we have the following lattice-theoretic characterization of a $p$-GSS. We recall that the \emph{elementary divisors} (a.k.a. \emph{invariant factors}) of an integral matrix are the nonzero entries of the \emph{Smith normal form} of the matrix; see \S\ref{sec:generating} for more.

\begin{theorem}[proved in \S\ref{sec:generating}]\label{pGSS-char}
The following statements are equivalent for a matrix $A\in \mathbb{Z}^{m\times n}$ of rank~$r$ and a prime $p$:
(1) the columns of $A$ form a $p$-GSS,
(2) the rows of $A$ form a $p$-GSS,
(3) whenever $y^\top A$ and $Ax$ are integral, then $y^\top Ax$ is a $p$-adic rational,
(4) every elementary divisor of $A$ is a power of $p$,
(5) the GCD of the subdeterminants of $A$ of order $r$ is a power of $p$,
(6) there exists a matrix $B$ with $p$-adic entries such that $ABA=A$.
\end{theorem}

\Cref{pGSS-char} is used in \S\ref{sec:generating} to prove that testing the $p$-GSS property can be done in polynomial time. Subsequently, the problem of testing total dual $p$-adicness belongs to co-NP by \Cref{TDp-char} (see \S\ref{sec:TDp}), and the problem of testing the $p$-GSC property belongs to co-NP by \Cref{pGSC-char} (see \S\ref{sec:generating}). Whether the two problems belong to NP, or P, remains unsolved. It should be pointed out that testing total dual \emph{integrality}, as well as testing the Hilbert basis property, is co-NP-complete~\cite{Ding08,Pap11}.

\subsection{Connection to integral polyhedra and TDI systems}

Our characterizations stated so far, as well as our characterization of \emph{$p$-adic polyhedra} explained in \S\ref{sec:polyhedra}, have the following intriguing consequence:

\begin{theorem}\label{TDpq}
Let $A\in \mathbb{Z}^{m\times n},b\in \mathbb{Z}^m$, and $P:=\{x:Ax\leq b\}$. 
Then the following are equivalent:
(1) $Ax\leq b$ is totally dual $p$-adic for all primes $p$,
(2) $Ax\leq b$ is totally dual $p$- and $q$-adic, for distinct primes $p,q$,
(3) for every nonempty face $F$ of $P$, the GCD of the subdeterminants of $A_F$ of order $\rank{A_F}$ is $1$. 
\end{theorem}
\begin{proof}
{\bf (1) $\Rightarrow$ (2)} is immediate.
{\bf (2) $\Rightarrow$ (3)} For every nonempty face $F$ of $P$, the rows of $A_F$ form both a $p$- and a $q$-GSS by \Cref{TDp-char}, so the GCD of the subdeterminants of $A_F$ of order $\rank{A_F}$ is both a power of $p$ and a power of $q$ by \Cref{pGSS-char}, so the GCD of the subdeterminants of $A_F$ of order $\rank{A_F}$ must be $1$.
{\bf (3) $\Rightarrow$ (1)} follows from \Cref{TDp-char} and \Cref{pGSS-char}
\end{proof}

If $Ax\leq b$ is TDI, and therefore totally dual $p$-adic for any prime $p$, then statement (3) above must hold (this is folklore, and explored in \cite{Sebo90}. In fact, if $P$ is pointed, then for every vertex of $P$, we have a stronger property known as \emph{local strong unimodularity}~\cite{Gerards87}.) It was a widely known fact that the converse is not true. \Cref{TDpq} clarifies this further by equating (3) with (1) and (2). Going a step further, it is known that if $Ax\leq b$ is TDI, then $\{x:Ax\leq b\}$ is an integral polyhedron~\cite{Edmonds77,Giles79}. We shall strengthen this result in the pointed case:

\begin{theorem}[proved in \S\ref{sec:polyhedra}]\label{pq->integral}
If $Ax\leq b$ is totally dual $p$- and $q$-adic, for distinct primes $p,q$, then $\{x:Ax\leq b\}$ is an integral polyhedron.
\end{theorem}

Fulkerson's theorem that every integral set packing system is TDI, can be seen as a (stronger) converse to \Cref{pq->integral}~\cite{Fulkerson71}. As for set covering systems, there is a conjecture of Paul Seymour that predicts a (stronger) converse to \Cref{pq->integral}.

\begin{CN}[The Dyadic Conjecture \cite{Schrijver03}, \S79.3e]\label{dyadic-CON}
Let $A$ be a matrix with $0,1$ entries. If $Ax\geq \1,x\geq \0$ defines an integral polyhedron, then it is TDD.
\end{CN}

The authors recently proved the first nontrivial step of the Dyadic Conjecture: If $Ax\geq 1,x\geq \0$ defines an integral polyhedron, then for every nonnegative integral $w$ such that $\min\{w^\top x:Ax\geq \1,x\geq \0\}$ has optimal value two, the dual has a dyadic optimal solution~\cite{Abdi-dyadic}.

\subsection{Examples}

Our first example comes from Integer Programming, and more precisely, from matrices with restricted subdeterminants. Similar settings to the one below have been studied previously; see for example \cite{Lee89,Appa04} (the last reference has more relevant citations).

\begin{theorem}\label{p-adic-matrix}
Let $A\in \mathbb{Z}^{m\times n}$ be a matrix whose subdeterminants belong to $\{0\}\cup \{\pm p^k:k\in \mathbb{Z}_+\}$ for some prime $p$, and let $b\in \mathbb{Z}^m$. Then $Ax\leq b$ is totally dual $p$-adic.
\end{theorem}
\begin{proof}
Choose an integral $w$ such that $\min\{b^\top y: A^\top y=w, y\geq\0\}$ has an optimal solution. Let $y^\star$ be a basic optimal solution. We claim that $y^\star$ is $p$-adic, thereby finishing the proof. Let $I:=\{i\in [m]:y^\star_i>0\}$. Then $A_I$, the row submatrix of $A$ corresponding to the indices $I$, has full row rank. In particular, $y^\star_I$ is the unique solution to $A_I^\top z=w$. By moving to a square row submatrix $B$ of $A_I^\top$ of full rank, and the corresponding subvector $w_B$ of $w$, we see that $y^\star_I$ is the unique solution to $B z=w_B$. Since $B$ is nonsingular, $\det(B)\in \{\pm p^k:k\in \mathbb{Z}_+\}$, so by Cramer's rule $y^\star_I$, and therefore $y^\star$, is $p$-adic.
\end{proof}

For example, the node-edge incidence matrix of a graph is known to satisfy the hypothesis for $p=2$ (folklore), and therefore leads to a TDD system. More generally, matrices whose subdeterminants belong to $\{0\}\cup \{\pm 2^k:k\in \mathbb{Z}_+\}$ have been studied from a matroid theoretic perspective; matroids representable over the rationals by such matrices are known as \emph{dyadic matroids} and their study was initiated by Whittle~\cite{Whittle95}.

Moving on, from Combinatorial Optimization, we get examples only in the dyadic setting. Let $G=(V,E)$ be a graph, and $T$ a nonempty subset of even cardinality. A \emph{$T$-join} is a subset $J\subseteq E$ such that the odd-degree vertices of $G[J]$ is precisely $T$. $T$-joins were studied due to their connection to the \emph{minimum weight perfect matching problem}, but also to the \emph{Chinese postman set problem} (see \cite{Cook98}, Chapter 5). As a consequence of a recent result~\cite{Abdi-Tjoins}, we shall obtain the following.

\begin{theorem}[proved in \S\ref{sec:examples}]\label{Tjoins-TDD}
Let $G=(V,E)$ be a graph, and $T\subseteq V$ a nonempty subset of even cardinality. Then the linear system $x(J)\geq 1~\forall \text{ $T$-joins $J$}; x\geq \0$ is TDD.
\end{theorem}

The basic solutions to the dual of $\min\{\1^\top x:x(J)\geq 1~\forall \text{ $T$-joins $J$}; x\geq \0\}$ may be non-dyadic, as we note in \S\ref{sec:examples}, thereby creating an interesting contrast between the proofs of \Cref{p-adic-matrix} and \Cref{Tjoins-TDD}. The system in \Cref{Tjoins-TDD} defines an integral set covering polyhedron (see \cite{Cornuejols01}, Chapter 2), so \Cref{Tjoins-TDD} verifies \Cref{dyadic-CON} for such instances. In fact, it has been conjectured that the system in \Cref{Tjoins-TDD} is totally dual \emph{quarter-integral} (\cite{Cornuejols01}, Conjecture 2.15). Observe that quarter-integrality is a stronger variant of dyadicness, and should not be confused with ``$4$-adicness'', which is not even defined in this paper.

Moving on, let $G=(V,E)$ be a graph. A \emph{cycle} is a subset $C\subseteq E$ such that every vertex in $V$ is incident with an even number of edges in $C$. Observe that every cycle yields an Eulerian subgraph. A \emph{circuit} is a nonempty cycle that does not contain another nonempty cycle. A \emph{perfect matching} is a subset $M\subseteq E$ such that every vertex in $V$ is incident with exactly one edge in $M$. Define $\mathbf{C}(G):=\{\chi_C:C \text{ a circuit of $G$}\}$ and $\mathbf{M}(G):=\{\chi_M:M \text{ a perfect matching of $G$}\}$. See \cite{Goddyn93} for an excellent survey on lattice and conic characterizations of these two sets.\footnote{The reader should take caution as \cite{Goddyn93} uses a different definition of Hilbert bases, following the convention set forth by~\cite{Schrijver03}.} 
We shall prove the following:

\begin{theorem}[proved in \S\ref{sec:examples}]\label{circuits-DGSC}
Let $G=(V,E)$ be a graph. Then $\mathbf{C}(G)$ is a DGSC.
\end{theorem}

If $G$ is bridgeless, then the \emph{Cycle Double Cover Conjecture}~\cite{Szekeres73,Seymour81} predicts that $\1$ can be written as a \emph{half-integral} conic combination of the vectors in $\mathbf{C}(G)$; \Cref{circuits-DGSC} implies this can be done \emph{dyadically}.

\begin{theorem}[proved in \S\ref{sec:examples}]\label{pm-DGSC}
Let $G=(V,E)$ be a graph such that $|V|$ is even. Then $\mathbf{M}(G)$ is a DGSC.
\end{theorem}

If $G$ is an \emph{$r$-graph}, then the \emph{Generalized Berge-Fulkerson Conjecture}~\cite{Seymour79} predicts that $\1$ can be written as a \emph{half-integral} conic combination of $\mathbf{M}(G)$; \Cref{pm-DGSC} proves this can be done \emph{dyadically}.

In \S\ref{sec:examples}, we see that the preceding three theorems do \emph{not} extend to the $p$-adic setting for $p\geq 3$, further emphazing the importance of the dyadic setting and the focus of the paper.

\section{Density Lemma and the Theorem of the Alternative}\label{sec:density}

Many of our results are made possible by an important feature of the $p$-adic rationals distinguishing them from the integers, namely \emph{density}.

\begin{RE}\label{density-remark}
The $p$-adic rationals form a dense subset of $\mathbb{R}$.
\end{RE}

\begin{LE}[Density Lemma]\label{density}
Let $A\in \mathbb{Z}^{m\times n},b\in \mathbb{Z}^m$.
If $\{x:Ax=b\}$ contains a $p$-adic point, then the $p$-adic points in the set form a dense subset. In particular, a nonempty rational polyhedron contains a $p$-adic point if, and only if, its affine hull contains a $p$-adic point.
\end{LE}
\begin{proof}
It suffices to prove the first statement. Suppose $\{x:Ax=b\}$ contains a $p$-adic point, say $\hat{x}$. Since $A$ has integral entries, its kernel has an integral basis, say $d^1,\ldots,d^r$. Observe that $\{x:Ax=b\}$ is the set of vectors of the form $\hat{x}+\sum_{i=1}^r \lambda_i d^i$ where $\lambda\in \mathbb{R}^r$. Consider the set $$
S:=\left\{\hat{x}+\sum_{i=1}^r \lambda_i d^i: \lambda_i \text{ is $p$-adic for each $i$}\right\}.
$$ By \Cref{density-remark}, it can be readily checked that $S$ is a dense subset of $\{x:Ax=b\}$. Since $\hat{x}$ is $p$-adic, and the $d^i$'s are integral, the points in $S$ are $p$-adic, thereby proving the lemma.
\end{proof}

A natural follow-up question arises: When does a rational subspace contain a $p$-adic point? Addressing this question requires a familiar notion in Integer Programming. Every integral matrix of full row rank can be brought into \emph{Hermite normal form} by means of \emph{elementary unimodular column operations}. In particular, if $A$ is an integral $m\times n$ matrix of full row rank, there exists an $n\times n$ \emph{unimodular} matrix $U$ such that $AU=(B~\0)$, where $B$ is a non-singular $m\times m$ matrix, and $\0$ is an $m\times (n-m)$ matrix with zero entries. By a square unimodular matrix, we mean a square integral matrix whose determinant is $\pm 1$; note that the inverse of such a matrix is also unimodular. See (\cite{ConfortiCornuejolsZambelli2014}, Section 1.5.2) or (\cite{Schrijver98}, Chapter 4) for more details.

\begin{LE}[Theorem of the Alternative]\label{alternative}
Let $A\in \mathbb{Z}^{m\times n}, b\in \mathbb{Z}^m$ and $p$ a prime. Then either $Ax=b$ has a $p$-adic solution, or there exists a $y\in \mathbb{R}^m$ such that $y^\top A$ is integral and $y^\top b$ is non-$p$-adic, but not both.
\end{LE}
\begin{proof}
Suppose $A\hat{x}=b$ for a $p$-adic point $\hat{x}$, and $y^\top A$ is integral. Then $y^\top b=y^\top (A\hat{x}) = (y^\top A)\hat{x}$ is an integral linear combination of $p$-adic rationals, and is therefore a $p$-adic rational. Thus, both statements cannot hold simultaneously.
Suppose $Ax=b$ has no $p$-adic solution. If $Ax=b$ has no solution at all, then there exists a vector $y$ such that $y^\top A=\0$ and $y^\top b\neq 0$; by scaling $y$ appropriately, we can ensure that $y^\top b$ is non-$p$-adic, as desired. Otherwise, $Ax=b$ has a solution. We may assume that $A$ has full row rank. Then there exists a square unimodular matrix $U$ such that $AU = (B~\0)$, where $B$ is a non-singular matrix. Observe that $\{x:Ax=b\}=\{Uz:AUz=b\}$. Thus, as $Ax=b$ has no $p$-adic solution $x$, and $U$ has integral entries, we may conclude that the system $AUz=b$ has no $p$-adic solution $z$ either. Let us expand the latter system. Let $I,J$ be the sets of column labels of $B,\0$ in $AU=(B~\0)$, respectively. Then
$$
\left\{z:AUz=b\right\}
=\left\{z:(B~\0)\begin{pmatrix}z_I\\z_J\end{pmatrix}=b\right\}
=\left\{z:Bz_I = b, z_J \text{ free}\right\}
=\left\{z:z_I=B^{-1}b, z_J \text{ free}\right\}.
$$
In particular, since $AUz=b$ has no $p$-adic solution, the vector $B^{-1}b$ is non-$p$-adic. Thus, there exists a row $y^\top$ of $B^{-1}$ for which $y^\top b$ is non-$p$-adic. We claim that $y^\top A$ is integral, thereby showing $y$ is the desired vector. To this end, observe that $B^{-1} AU = B^{-1} (B~\0) = (I~\0)$, implying in turn that $B^{-1}A = (I~\0)U^{-1}$. As the inverse of a square unimodular matrix, $U^{-1}$ is also unimodular and therefore has integral entries, implying in turn that $B^{-1}A$, and in particular $y^\top A$, is integral. 
\end{proof}

The reader may notice a similarity between \Cref{alternative} and its integer analogue, which characterizes when a linear system of equations admits an integral solution, commonly known as the \emph{Integer Farkas Lemma} (see \cite{ConfortiCornuejolsZambelli2014}, Theorem 1.20). We refrain from calling \Cref{alternative} the ``$p$-adic Farkas Lemma'' as we reserve that title for \Cref{pFL} below.

\begin{RE}\label{RE:pq->integral}
If $t$ is a $p$- and $q$-adic rational, for distinct primes $p,q$, then $t$ is integral.
\end{RE}

\begin{CO}\label{CO:pq->integral}
Let $A\in \mathbb{Z}^{m\times n}, b\in \mathbb{Z}^m$. If $Ax=b$ has $p$- and $q$-adic solutions, for distinct primes $p$ and $q$, then the system has an integral solution.
\end{CO}
\begin{proof}
By the Theorem of the Alternative, whenever $y^\top A$ is integral, $y^\top b$ is both $p$- and $q$-adic, implying in turn that $y^\top b$ is integral by \Cref{RE:pq->integral}. Thus, by the Integer Farkas Lemma, $Ax=b$ has an integral solution.
\end{proof}

Finally, the Density Lemma and the Theorem of the Alternative have the following $p$-adic analogue of Farkas Lemma in Linear Programming.

\begin{CO}[$p$-Adic Farkas Lemma]\label{pFL}
Let $P$ be a nonempty rational polyhedron whose affine hull is $\{x:Ax=b\}$, where $A,b$ are integral. Then for every prime $p$, $P$ contains a $p$-adic point if, and only if, there does not exist $y$ such that $y^\top A$ is integral and $y^\top b$ is non-$p$-adic.
\end{CO}

\section{$p$-Adic generating sets for subspaces and cones}\label{sec:generating}

Recall that a set of vectors $\{a^1,\ldots,a^n\}\subseteq \mathbb{Z}^m$ forms a $p$-GSS if every integral vector in the linear hull of the vectors can be expressed as a $p$-adic linear combination of the vectors. Observe that every $p$-adic vector in the linear hull of a $p$-GSS can also be expressed as a $p$-adic linear combination of the vectors.

\begin{LE}\label{unimodular-transformations}
Let $A\in \mathbb{Z}^{m\times n}$, and $U$ a unimodular matrix of appropriate dimensions. Then (1) the columns of $A$ form a $p$-GSS if, and only if, the columns of $UA$ do, and (2) the columns of $A$ form a $p$-GSS if, and only if, the columns of $AU$ do.
\end{LE}
\begin{proof}
{\bf (1)}
It suffices to prove $(\Rightarrow)$, since $U^{-1}$ is also unimodular. Pick $b\in \mathbb{Z}^m$ such that $(UA)x=b$ has a solution $\bar{x}$; we need to show that the system has a $p$-adic solution. Note that $A\bar{x}=U^{-1}b\in \mathbb{Z}^m$, so by assumption, there exists a $p$-adic $x^\star$ such that $Ax^\star=U^{-1}b$. Left-multiplying by $U$, we get that $(UA)x^\star=b$, so $x^\star$ is the desired $p$-adic solution. {\bf (2)} Once again, it suffices to prove $(\Rightarrow)$. Pick $b\in \mathbb{Z}^m$ such that $(AU)z=b$ has a solution $\bar{z}$; we need to show that the system has a $p$-adic solution. In particular, $Ax=b$ has a solution, namely $U\bar{z}$, so by assumption, there exists a $p$-adic $x^\star$ such that $Ax^\star=b$. Let $z^\star=U^{-1}x^\star$, which is also $p$-adic. Then $(AU)z^\star=Ax^\star=b$, so $z^\star$ is the desired $p$-adic solution.
\end{proof}

In order to prove \Cref{pGSS-char}, we need a definition. Let $A$ be an integral matrix of rank $r$. It is well-known that by applying elementary row \emph{and} column operations, we can bring $A$ into \emph{Smith normal form}, that is, into a matrix with a leading $r\times r$ minor $D$ and zeros everywhere else, where $D$ is a diagonal matrix with diagonal entries $\delta_1,\ldots,\delta_r\geq 1$ such that $\delta_1 \mid \delta_2 \mid \cdots \mid \delta_r$ (see \cite{Schrijver98}, Section 4.4). It can be readily checked that for each $i\in [r]$, $\prod_{j=1}^{i}\delta_j$ is the GCD of the subdeterminants of $A$ of order~$i$. The $\delta_i$'s are referred to as the \emph{elementary divisors}, or \emph{invariant factors}, of $A$. The Smith normal form of an integral matrix, and therefore its elementary divisors, can be computed in polynomial time~\cite{Kannan79}.
%
\begin{proof}[Proof of \Cref{pGSS-char}]
{\bf (1) $\Leftrightarrow$ (3)} 
Suppose (1) holds. Choose $x,y$ such that $y^\top A$ and $Ax$ are integral. Let $b:=Ax\in \mathbb{Z}^m$. By (1), there exists a $p$-adic $\bar{x}$ such that $b=A\bar{x}$. Thus, $y^\top Ax=y^\top A\bar{x}=(y^\top A)\bar{x}$, which is $p$-adic because $y^\top A$ is integral and $\bar{x}$ $p$-adic, as required. Suppose conversely that (3) holds. Pick $b\in \mathbb{Z}^m$ such that $A\bar{x}=b$ for some $\bar{x}$. We need to prove that $Ax=b$ has a $p$-adic solution. If $y^\top A$ is integral, then $y^\top b= y^\top A\bar{x}$, which is $p$-adic by (3). Thus, by the Theorem of the Alternative, $Ax=b$ has a $p$-adic solution, as required.

{\bf (2) $\Leftrightarrow$ (3)} holds by applying the established equivalence (1) $\Leftrightarrow$ (3) to $A^\top$.

{\bf (1)-(3) $\Leftrightarrow$ (4)}: By \Cref{unimodular-transformations}, the equivalent conditions (1)-(3) are preserved under elementary unimodular row/column operations; these operations clearly preserve (4) as well. Thus, it suffices to prove the equivalence between (1)-(3) and (4) for integral matrices in Smith normal form. That is, we may assume that $A$ has a leading $r\times r$ minor $D$ and zeros everywhere else, where $D$ is a diagonal matrix with diagonal entries $\delta_1,\ldots,\delta_r\geq 1$ such that $\delta_1 \mid \delta_2 \mid \cdots \mid \delta_r$. 
Suppose (1)-(3) hold. We need to show that each $\delta_i$ is a power of $p$. Consider the feasible system $Ax=e_i$; every solution $x$ to this system satisfies $x_j=0, j\in [r]-\{i\}$ and $x_i=\frac{1}{\delta_i}$. Since the columns of $A$ form a $p$-GSS, $\frac{1}{\delta_i}$ must be $p$-adic, so $\delta_i$ is a power of $p$, as required.
Suppose conversely that (4) holds. We need to show that whenever $Ax=b,b\in \mathbb{Z}^m$ has a solution, then it has a $p$-adic solution. Clearly, it suffices to prove this for $b=e_i,i\in [r]$, which holds because each $\delta_i,i\in [r]$ is a power of $p$.

{\bf (4) $\Leftrightarrow$ (5)} is rather immediate; the only additional remark is that every divisor of a power of $p$ is also a power of $p$.

{\bf (6) $\Rightarrow$ (3)} If $y^\top A$ and $Ax$ are integral, then $y^\top Ax=y^\top (ABA)x=(y^\top A)B(Ax)$, which is $p$-adic since $y^\top A,Ax$ are integral and $B$ has $p$-adic entries, as required.

{\bf (4) $\Rightarrow$ (6)} Choose unimodular matrices $U,W$ such that $UAW$ is in Smith normal form with elementary divisors $\delta_1,\ldots,\delta_r$. Let $B'$ be the $n\times m$ matrix with a leading diagonal matrix $D^{-1}=\Diag(\frac{1}{\delta_1},\ldots,\frac{1}{\delta_r})$, and zeros everywhere else. Let $B:=WB'U$, which is a matrix with $p$-adic entries since each $\delta_i$ is a power of $p$. We claim that $ABA=A$, thereby proving (6). This equality holds if, and only if, $UABAW=UAW$. To this end, we have $$
UABAW=UA(WB'U)AW=(UAW)B'(UAW)=UAW
$$ where the last equality holds due to the definition of $B'$ and the Smith normal form of $UAW$.
\end{proof}

In light of the previous proposition we may say that an integral \emph{matrix} forms a $p$-GSS if its rows, respectively its columns, form a $p$-GSS. Consider the following complexity problem: \emph{$(A)$ Given an integral matrix, does it form a $p$-GSS?}

\begin{theorem}\label{pGSS-P}
$(A)$ belongs to $\mathrm{P}$.
\end{theorem}
\begin{proof}
Let $A$ be an integral matrix. By \Cref{pGSS-char}, $A$ forms a $p$-GSS if, and only if, each elementary divisor of $A$ is a power of $p$. Since the elementary divisors of $A$ can be found in polynomial time, the theorem follows.
\end{proof}

Recall that a set of vectors $\{a^1,\ldots,a^n\}\subseteq \mathbb{Z}^m$ forms a $p$-GSC if every integral vector in the conic hull of the vectors can be expressed as a $p$-adic conic combination of the vectors. Observe that every $p$-adic vector in the conic hull of a $p$-GSC can also be expressed as a $p$-adic conic combination of the vectors.

\begin{PR}\label{pGSC->pGSS}
If $\{a^1,\ldots,a^n\}\subseteq \mathbb{Z}^m$ is a $p$-GSC, then it is a $p$-GSS.
\end{PR}
\begin{proof}
Let $A\in \mathbb{Z}^{m\times n}$ be the matrix whose columns are $a^1,\ldots,a^n$. Take $b\in \mathbb{Z}^m$ such that $A\bar{x}=b$ for some $\bar{x}$. We need to show that the system $Ax=b$ has a $p$-adic solution. To this end, let $\bar{x}':=\bar{x}-\lfloor \bar{x}\rfloor\geq \0$ and $b':=A\bar{x}'=b-A\lfloor \bar{x}\rfloor\in \mathbb{Z}^m$. Thus, $Ax=b',x\geq \0$ has a solution, namely $\bar{x}'$, so it has a $p$-adic solution, say $\bar{z}'$, as the columns of $A$ form a $p$-GSC. Let $\bar{z}:=\bar{z}'+\lfloor \bar{x}\rfloor$, which is also $p$-adic. Then $A\bar{z}=A\bar{z}'+A\lfloor \bar{x}\rfloor=b'+A\lfloor \bar{x}\rfloor=b$, so $\bar{z}$ is a $p$-adic solution to $Ax=b$, as required.
\end{proof}

The converse of this result, however, does not hold. For example, let $k\geq 3$ be an integer, $n:=p^k+1$, and $m$ an integer in $\{4,\ldots,p^{k}\}$ such that $m-1$ is not a power of $p$. Consider the matrix 
$$
A:=\begin{pmatrix}
\begin{matrix}E_n-I_n\end{matrix}
&\rvline &
\begin{matrix}E_m-I_m\\ \hline \0\end{matrix}
\end{pmatrix}
$$ where $E_d,I_d$ denote the all-ones square and identity matrices of dimension $d$, respectively. We claim that the columns of $A$ form a $p$-GSS but not a $p$-GSC. To see the former, note that $A$ has rank $n$, and since $\det(E_n-I_n)=n-1=p^k$, the GCD of the subdeterminants of $A$ of order $n$ is a power of $p$, so the columns of $A$ form a $p$-GSS by \Cref{pGSS-char}. To see the latter, consider the vector $b\in \{0,1\}^n$ whose first $m$ entries are equal to $1$, and whose last $n-m$ entries are equal to $0$. Then $Ay=b,y\geq \0$ has a unique solution, namely $\bar{y}$ defined as $\bar{y}_i=0$ for $1\leq i\leq n$, and $\bar{y}_i=\frac{1}{m-1}$ for $n+1\leq i\leq n+m$. In particular, as $m-1$ is not a power of $p$, $b$ is an integral vector in the conic hull of the columns of $A$, but it cannot be expressed as a $p$-adic conic combination of the columns. Thus, the columns of $A$ do not form a $p$-GSC.\footnote{Smaller counterexamples exist; for instance, the vectors $(0,1,1,1,1),(1,0,1,1,1),(1,1,0,1,1),(1,1,1,0,1),$ $(1,1,1,1,0),(3,0,1,1,1)$ form a DGSS but not a DGSC.}

However, we do have the following sort of converse.

\begin{RE}\label{pGSS->pGSC}
If $\{a^1,\ldots,a^n\}\subseteq \mathbb{Z}^m$ is a $p$-GSS, then $\{\pm a^1,\ldots,\pm a^n\}$ is a $p$-GSC.
\end{RE}

\begin{PR}\label{pGSC-face}
Let $\{a^1,\ldots,a^n\}\subseteq \mathbb{Z}^m$ be a $p$-adic generating set for a cone, and $F$ a nonempty face of the cone. Then $\{a^i:a^i\in F\}$ is a $p$-adic generating set for the cone $F$.
\end{PR}
\begin{proof}
Let $b$ be an integral vector in the face $F$. Since $b\in C$, we can write $b$ as a $p$-adic conic combination of the vectors in $\{a^1,\ldots,a^n\}$. However, since $b$ is contained in the face $F$, the conic combination can only assign nonzero coefficients to the vectors in $F$, implying in turn that $b$ is a $p$-adic conic combination of the vectors in $\{a^i:a^i\in F\}$. As this holds for every $b$, $\{a^i:a^i\in F\}$ forms a $p$-GSC.
\end{proof}

\begin{proof}[Proof of \Cref{pGSC-char}]
$(\Rightarrow)$ follows from \Cref{pGSC-face} and \Cref{pGSC->pGSS}.
$(\Leftarrow)$ Let $b$ be an integral vector in $C$, and $F$ the minimal face of $C$ containing $b$. Let $B$ be the matrix whose columns are the vectors $\{a^i:a^i\in F\}$. We need to show that $Q:=\{y: By=b,y\geq \0\}$, which is nonempty, contains a $p$-adic point. By the Density Lemma, it suffices to show that $\aff(Q)$, the affine hull of $Q$, contains a $p$-adic point. Our minimal choice of $F$ implies that $Q$ contains a point $\mathring{y}$ such that $\mathring{y}>\0$, implying in turn that $\aff(Q)=\{y:By=b\}$. As the columns of $B$ form a $p$-GSS, and $b$ is integral, it follows that $\aff(Q)$ contains a $p$-adic point, as required.
\end{proof}

Consider the following complexity problem: \emph{$(B)$ Given a set of vectors, does it form a $p$-GSC?}

\begin{theorem}\label{pGSC-co-NP}
$(B)$ belongs to co-NP.
\end{theorem}
\begin{proof}
Suppose a set of integral vectors $\{a^1,\ldots,a^n\}$, whose conic hull is denoted $C$, does not a $p$-GSC. By \Cref{pGSC-char}, there exists a face $F=C\cap \{x:a^\top x\geq 0\}$ such that $S:=F\cap \{a^1,\ldots,a^n\}$ does not form a $p$-GSS. The subset $S$, along with the face provided by its supporting hyperplane $a^\top x\geq 0$, can be provided as a certificate that $\{a^1,\ldots,a^n\}$ is not a $p$-GSC. Testing that $S$ is not a $p$-GSS can be done in polynomial time by \Cref{pGSS-P}, so the result follows.
\end{proof}

\section{Totally dual $p$-adic systems}\label{sec:TDp}

Given integral $A,b$, recall that $Ax\leq b$ is totally dual $p$-adic if for every integral $w$ for which $\min\{b^\top y:A^\top y=w,y\geq\0\}$ has an optimal solution, it has a $p$-adic optimum. It can be readily checked that the rows of $A$ form a $p$-GSS if, and only if, $Ax= \0$ is totally dual $p$-adic; and the rows of $A$ form a $p$-GSC if, and only if, $Ax\leq \0$ is totally dual $p$-adic.

\begin{proof}[Proof of \Cref{TDp-char}]
Consider the following pair of dual linear programs, for $w$ later specified.
\begin{align}
&\max\{w^\top x: Ax\leq b\} \tag{P}\label{P-tdd-2} \\
&\min\{b^\top y: A^\top y=w, y\geq\0\} \tag{D}\label{D-tdd-2}.
\end{align}
For every nonempty face $F$ of $P$, denote by $A_{\bar{F}}$ the row submatrix of $A$ corresponding to the rows not in $A_F$. For every vector $y$, denote by $y_F,y_{\bar{F}}$ the variables corresponding to the rows in $A_F,A_{\bar{F}}$, respectively.

{\bf (1) $\Rightarrow$ (2)}
Consider a nonempty face $F$ of $P$. We need to show that the rows of $A_F$ form a $p$-GSC. Let $w$ be an integral vector in the conic hull of the rows of $A_F$. It suffices to express $w$ as a $p$-adic conic combination of the rows of $A_F$. To this end, observe that every point in $F$ is an optimal solution to \eqref{P-tdd-2}. As $Ax\leq b$ is TDD, \eqref{D-tdd-2} has a $p$-adic optimal solution, say $\bar{y}\geq \0$. As Complementary Slackness holds for all pairs $(\bar{x},\bar{y}), \bar{x}\in F$, it follows that $\bar{y}_{\bar{F}}=\0$. Subsequently, we have $w=A^\top \bar{y} = A_F^\top \bar{y}_F$, thereby achieving our objective. 
{\bf (2) $\Rightarrow$ (3)} follows from \Cref{pGSC->pGSS}.
{\bf (3) $\Rightarrow$ (1)}
Choose an integral $w$ for which \eqref{D-tdd-2} has an optimal solution; we need to show now that it has a $p$-adic optimal solution. Denote by $F$ the face of the optimal solutions to the primal linear program \eqref{P-tdd-2}. By Complementary Slackness, the set of optimal solutions to the dual \eqref{D-tdd-2} is $Q:=\{y:A^\top y=w, y\geq \0, y_{\bar{F}}=\0\}$. We need to show that $Q$ contains a $p$-adic point. In fact, by the Density Lemma, it suffices to find a $p$-adic point in $\aff(Q)$, the affine hull of $Q$. By Strict Complementarity, $Q$ contains a point $\mathring{y}$ such that $\mathring{y}_F>\0$, implying in turn that $\aff(Q)=\{y:A_F^\top y_F=w, y_{\bar{F}}=\0\}$. Since the rows of $A_F$ form a $p$-GSS, and $w$ is integral, we get that $\aff(Q)$ contains a $p$-adic point, as required.
\end{proof}

The careful reader may notice that by applying polarity to \Cref{TDp-char} with $b=\0$, we obtain another proof of \Cref{pGSC-char}. Moving on, consider the following complexity problem: \emph{$(C)$ Given a system $Ax\leq b$ where $A,b$ are integral, is the system totally dual $p$-adic?}

\begin{theorem}
$(C)$ belongs to co-NP.
\end{theorem}
\begin{proof}
Suppose $Ax\leq b$ is not totally dual $p$-adic for some integral $A,b$. Let $P:=\{x:Ax\leq b\}$.
By \Cref{TDp-char}, there exists a nonempty face $F:=P\cap \{x:a^\top x\leq \beta\}$ of $P$ such that the rows of $A_F$ do not form a $p$-GSS. The rows $A_F$, along with the face provided by the supporting hyperplane $a^\top x\leq \beta$, can be given as a certificate that $Ax\leq b$ is not totally dual $p$-adic. Testing that the rows of $A_F$ do not form a $p$-GSS can be done in polynomial time by \Cref{pGSS-P}, so the result follows.
\end{proof}

\section{$p$-Adic polyhedra}\label{sec:polyhedra}

A nonempty rational polyhedron is \emph{$p$-adic} if every nonempty face contains a $p$-adic point. In this section we provide a characterization of $p$-adic polyhedra.

\begin{RE}\label{free-non-neg}
Let $A\in \mathbb{Z}^{m\times n},b\in \mathbb{Z}^m,y\in\mathbb{R}^m$ and $y':=y-\lfloor y\rfloor\geq\0$. Then 
$A^\top y\in \mathbb{Z}^n$ if and only if $A^\top y'\in \mathbb{Z}^n$,
$y$ is $p$-adic if and only if $y'$ is $p$-adic, and
$b^\top y$ is $p$-adic if and only if $b^\top y'$ is $p$-adic.
\end{RE}

\begin{theorem}\label{p-adic-polyhedra}
Let $A\in \mathbb{Z}^{m\times n},b\in \mathbb{Z}^m$ and
$P:=\{x:Ax\leq b\}$. Then the following are equivalent for a prime $p$:
(1) $P$ is a $p$-adic polyhedron,
(2) for every nonempty face $F$ of $P$, $\aff(F)$ contains a $p$-adic point,
(3) for every nonempty face $F$ of $P$, and $z$, if $A_F^\top z$ is integral then $b_F^\top z$ is $p$-adic,
(4) for all $w\in\mathbb{R}^n$ for which $\max\{w^\top x:x\in P\}$ has an optimum, it has a $p$-adic optimal solution,
(5) for all $w\in\mathbb{Z}^n$ for which $\max\{w^\top x:x\in P\}$ has an optimum, it has a $p$-adic optimal value.
\end{theorem}

There is an intriguing contrast between this characterization and that of integral polyhedra~(see \cite{ConfortiCornuejolsZambelli2014}, Theorem 4.1), namely the novelty of statements (2) and (3), which are ultimately due to Strict Complementarity and the Density Lemma. 

\begin{proof}
{\bf (1) $\Rightarrow$ (2)} follows immediately from definition.
{\bf (2) $\Rightarrow$ (1)} By the Density Lemma, every nonempty face contains a $p$-adic point, so $P$ is a $p$-adic polyhedron.
{\bf (2) $\Leftrightarrow$ (3)} follows from the Theorem of the Alternative.
{\bf (1) $\Rightarrow$ (4)} Suppose $\max\{w^\top x:x\in P\}$ has an optimum. Let $F$ be the set of optimal solutions.
As $F$ is in fact a face of $P$, and $P$ is $p$-adic, it follows that $F$ contains a $p$-adic point.
{\bf (4) $\Rightarrow$ (5)} If $x$ is a $p$-adic vector, and $w$ an integral vector, then $w^\top x$ is a $p$-adic rational.

{\bf (5) $\Rightarrow$ (3)} We prove the contrapositive. Suppose (3) does not hold, that is, there exist a nonempty face $F$ and $z$ such that $w:=A_F^\top z\in \mathbb{Z}^n$ and $b_F^\top z$ is not $p$-adic. By \Cref{free-non-neg}, we may assume that $z\geq\0$. 
Consider the following pair of dual linear programs:
\begin{align}
&\max\{w^\top x: Ax\leq b\} \tag{P}\label{P-dyadic} \\
&\min\{b^\top y: A^\top y=w, y\geq\0\} \tag{D}\label{D-dyadic}
\end{align}
Denote by $A_{\bar{F}}$ the row submatrix of $A$ corresponding to rows not in $A_F$.
Denote by $y_F, y_{\bar{F}}$ the variables of \eqref{D-dyadic} corresponding to rows $A_F$ and $A_{\bar{F}}$ of $A$, respectively.
Define $\bar{y}\geq \0$ where $\bar{y}_F=z$ and $\bar{y}_{\bar{F}}=\0$.
Then $A^\top \bar{y}=A_F^\top z =w$, so $\bar{y}$ is feasible for \eqref{D-dyadic}.
Moreover, Complementary Slackness holds for every pair $(x,\bar{y}),x\in F$. Subsequently, $\bar{y}$ is an optimal solution to \eqref{D-dyadic}, and $b^\top \bar{y}=b_F^\top z$ is the common optimal value of the two linear programs. Since $w$ is integral and $b_F^\top z$ is not $p$-adic, (5) does not hold, as required.
\end{proof}

\begin{CO}\label{TDp->p-adic}
Let $A\in \mathbb{Z}^{m\times n},b\in \mathbb{Z}^m$, and $p$ a prime. If $Ax\leq b$ is totally dual $p$-adic, then $\{x:Ax\leq b\}$ is a $p$-adic polyhedron.
\end{CO}
\begin{proof}
This follows immediately from \Cref{p-adic-polyhedra}~(5) $\Rightarrow$ (1).
\end{proof}

\begin{proof}[Proof of \Cref{pq->integral}]
By \Cref{TDp->p-adic}, $P:=\{x:Ax\leq b\}$ is a $p$- and $q$-adic polyhedron, that is, every minimal nonempty face of $P$ contains a $p$-adic point and a $q$-adic point. Each minimal nonempty face of $P$ is an affine subspace, so by \Cref{CO:pq->integral}, it contains an integral point. Thus, every minimal nonempty face of $P$ contains an integral point, so $P$ is an integral polyhedron.
\end{proof}

\section{$T$-joins, circuits, and perfect matchings}\label{sec:examples}

Let $G=(V,E)$ be a graph, and $T$ a nonempty subset of even cardinality. A \emph{$T$-cut} is a cut of the form $\delta(U)$ where $|U\cap T|$ is odd. Recall that a $T$-join is a subset $J\subseteq E$ such that the set of odd-degree vertices of $G[J]$ is precisely $T$. It can be readily checked that every $T$-cut and $T$-join intersect (see \cite{Cornuejols01}, Chapter~2). The following result was recently proved:

\begin{theorem}[\cite{Abdi-Tjoins}]\label{Tjoins-dyadic}
Let $G=(V,E)$ be a graph, and $T$ a nonempty subset of even cardinality. Let $\tau$ be the minimum cardinality of a $T$-cut. Then there exists a dyadic assignment $y_J\geq \0$ to every $T$-join~$J$ such that $\1^\top y=\tau$ and $\sum\left(y_J:J\text{ a $T$-join containing $e$}\right)\leq 1~\forall e\in E$.
\end{theorem}

The proof of \Cref{Tjoins-dyadic} uses the Density Lemma, the Theorem of the Alternative, and a result of Lov\'{a}sz on the \emph{matching lattice}~\cite{Lovasz87}.

\begin{proof}[Proof of \Cref{Tjoins-TDD}]
Let $A$ be the matrix whose columns are labeled by $E$, and whose rows are the incidence vectors of the $T$-joins. We need to show that $\min\{w^\top x: Ax\geq \1,x\geq \0\}$ yields a TDD system. Choose an integral $w$ such that the dual $\max\{\1^\top y: A^\top y\leq w, y\geq\0\}$ has an optimal solution, that is, $w\geq \0$. Let $G'$ be obtained from $G$ after replacing every edge $e$ with $w_e$ parallel edges (if $w_e=0$, then $e$ is deleted). Let $\tau_w$ be the minimum cardinality of a $T$-cut of $G'$, which is also the minimum weight of a $T$-cut of $G$. By \Cref{Tjoins-dyadic}, there exists a dyadic assignment $\bar{y}_J\geq 0$ to every $T$-join of $G'$ such that $\1^\top \bar{y}=\tau_w$ and $\sum\left(\bar{y}_J:J\text{ a $T$-join of $G'$ containing $e$}\right)\leq 1~\forall e\in E(G')$. This naturally gives a dyadic assignment $y^\star_J\geq 0$ to every $T$-join of $G$ such that $\1^\top y^\star=\tau_w$ and $A^\top y^\star\leq w$. Now let $\delta(U)$ be a minimum weight $T$-cut of $G$. Then $\chi_{\delta(U)}$ is a feasible solution to the primal which has value $\tau_w$. As a result, $\chi_{\delta(U)}$ is optimal for the primal, and $y^\star$ is optimal for the dual. Thus, the dual has a dyadic optimal solution, as required.
\end{proof}

Given the proof of \Cref{p-adic-matrix}, a natural attempt to prove \Cref{Tjoins-TDD} is to prove that every basic optimal solution of $\max\{\1^\top y: A^\top y\leq w, y\geq\0\}$ is dyadic. However, this is not necessarily true if $G=(V,E)$ is a \emph{snark}, $T=V$, and $w=\1$. In fact, as was shown by Palion~\cite{Palion-thesis}, for the majority of snarks on up to $40$ vertices, an off-the-shelf LP solver finds a basic optimal solution that is not dyadic. For example, if $G$ is the first \emph{Blanu\v{s}a} snark with HoG graph ID 2760~\cite{HoG}, which has $18$ vertices, the linear program has a basic optimal solution that assigns $\frac13$ to some seven perfect matchings, and $\frac23$ to some other perfect matching. 

\Cref{Tjoins-TDD} does not extend to the $p$-adic setting for any prime $p\geq 3$. To see this, let $G$ be the graph with vertices $1,2,3,4,5$ and edges $\{1,3\},\{1,4\},\{1,5\},\{3,2\},\{4,2\},\{5,2\}$, let $T:=\{1,2,3,4\}$, and let $w:=\1$. Then the linear program $\max\{\1^\top y: A^\top y\leq w, y\geq\0\}$ has a unique optimal solution, namely $y^\star=\frac12\cdot \1$, which is not $p$-adic for any $p\geq 3$.

Moving on, let $G=(V,E)$ be a graph such that $|V|$ is even. Let us prove that $\mathbf{M}(G)$, which is equal to the set $\{\chi_M:M \text{ a perfect matching of $G$}\}$, is a DGSC.

\begin{proof}[Proof of \Cref{pm-DGSC}]
We may assume that $G$ contains a perfect matching. Let $T:=V$. Note that every $T$-join has cardinality at least $\frac{|V|}{2}$, with equality holding precisely for the perfect matchings. By \Cref{Tjoins-TDD}, the linear system $x(J)\geq 1~\forall \text{ $T$-joins $J$}; x\geq \0$ is TDD. Let $P$ be the corresponding polyhedron, and $F$ the minimal face containing the point $\frac{2}{|V|}\cdot \1$. The tight constraints of $F$ are precisely $x(M)\geq 1$ for perfect matchings $M$, so by \Cref{TDp-char} for $p=2$, the rows of the corresponding coefficient matrix form a DGSC, implying in turn that $\mathbf{M}(G)$ is a DGSC.
\end{proof}

Let $P_{10}$ be the Petersen graph. Then $P_{10}$ has six perfect matchings. Let $M$ be the matrix whose columns are labeled by $E(P_{10})$, and whose rows are the incidence vectors of the perfect matchings. It can be checked that the elementary divisors of $M$ are $(1,1,1,1,1,2)$. Thus, for any prime $p\geq 3$, the rows of $M$ which are the vectors in $\mathbf{M}(P_{10})$ do not form a $p$-GSS by \Cref{pGSS-char}, and so they do not form a $p$-GSC by \Cref{pGSC->pGSS}.\footnote{If one followed the other definition of Hilbert bases as stated in \cite{Schrijver03}, then $\mathbf{M}(P_{10})$ would form a Hilbert basis, as noted in \cite{Goddyn93}.} Thus, \Cref{pm-DGSC} does not extend to the $p$-adic setting for $p\geq 3$.

Let $G=(V,E)$ be a graph. Recall that a cycle is a subset $C\subseteq E$ such that every vertex in $V$ is incident with an even number of edges in $C$. In particular, $\emptyset$ is a cycle. Define $\mathbf{C'}(G):=\{\chi_C:C \text{ a cycle of $G$}\}$. We shall prove that this set is a DGSC. First, we need to prove the following:

\begin{LE}\label{cycles-TDD}
Let $G=(V,E)$ be a graph. Consider edge variables $x,z\in \mathbb{R}^E$. Then the linear system $x(C)+z(E\setminus C)\geq 1 ~\forall \text{ cycles $C$}; x\geq \0; z\geq \0$ is TDD.
\end{LE}
\begin{proof}
Let $G'$ be the graph obtained from $G$ after subdividing every edge once, and let $T'$ be any subset of $V(G')$ that has even cardinality and contains all the new vertices (of degree two). Denote by $\{e,\hat{e}:e\in E\}$ the edge set of $G'$. It can be readily checked that the family of $T'$-joins of $G'$ is equal to $\{C\cup \hat{C}:C \text{ a cycle of $G$}\}$, after swapping $e,\hat{e}$ for some edges $e\in E$, if necessary. Subsequently, $y(J)\geq 1~\forall \text{ $T'$-joins $J$ of $G'$}; y\geq \0$ is just a relabeling of $x(C)+z(E\setminus C)\geq 1 ~\forall \text{ cycles $C$ of $G$}; x\geq \0; z\geq \0$. Since the first system is TDD by \Cref{Tjoins-TDD}, the second system is also TDD.
\end{proof}

\Cref{cycles-TDD} above takes advantage of the fact that the \emph{cuboid} of the cycle space of a graph is isomorphic to the clutter of $T$-joins of another graph~\cite{Abdi-cuboids}.

\begin{LE}\label{cycles-DGSC}
Let $G=(V,E)$ be a graph. Then $\mathbf{C'}(G)$ is a DGSC.
\end{LE}
\begin{proof}
Let $M$ be the coefficient matrix of the linear system $x(C)+z(E\setminus C)\geq 1 ~\forall \text{ cycles $C$}$ from \Cref{cycles-TDD}. Note that $M$ is a $0,1$ matrix with exactly $|E|$ $1$s per row. Let $A$ be the column submatrix corresponding to the $x$ variables; observe that $E-A$ is the column submatrix corresponding to the $z$ variables, where $E$ denotes the all-ones matrix. By \Cref{cycles-TDD}, $Ax+(E-A)z\geq \1;x\geq \0;z\geq \0$ is TDD. Let $P$ be the corresponding polyhedron, and $F$ the minimal face containing the point $\frac{1}{|E|}\cdot \1$. The tight constraints of $F$ are precisely $Ax+(E-A)z\geq \1$, so by \Cref{TDp-char} for $p=2$, the rows of $(A\mid E-A)$ form a DGSC. 

To finish the proof, we need to show that the rows of $A$, which are precisely the elements of $\mathbf{C'}(G)$, form a DGSC. To this end, let $w$ be an integral vector in the conic hull of the rows of $A$. Then there exists a $\bar{y}\geq 0$ such that $\bar{y}^\top A  = w^\top$. Observe that $\bar{y}^\top (A\mid E-A) = (w^\top \mid (\1^\top \bar{y})\cdot \1^\top -w^\top)\geq \0$. Since $A$ has a zero row, we may assume that $\1^\top \bar{y}$ is an integer. Subsequently, $\bar{y}^\top (A\mid E-A)$ is an integral vector, so because the rows of $(A\mid E-A)$ form a DGSC, there exists a dyadic $y'\geq \0$ such that $\bar{y}^\top (A\mid E-A)=y'^\top (A\mid E-A)$. In particular, $y'^\top A=w^\top$, so $w$ is a dyadic conic combination of the rows of $A$, as needed.
\end{proof}

Recall that a circuit is a nonempty cycle that does not contain another nonempty cycle. Recall that $\mathbf{C}(G)=\{\chi_C:C \text{ a circuit of $G$}\}$.

\begin{proof}[Proof of \Cref{circuits-DGSC}]
Observe that every nonempty cycle is the disjoint union of some circuits. Thus, since $\mathbf{C'}(G)$ is a DGSC by \Cref{cycles-DGSC}, we obtain immediately that $\mathbf{C}(G)$ is a DGSC.
\end{proof}

Let $H$ be the graph on two vertices with three parallel edges. Then $\1\in \cone(\mathbf{C}(H))$. In fact, $\1$ can be written as a unique conic combination of the vectors in $\mathbf{C}(H)$, in which every vector is assigned a coefficient of $\frac12$. Consequently, for any prime $p\geq 3$, the vectors in $\mathbf{C}(H)$ do not form a $p$-GSC, so \Cref{circuits-DGSC} does not extend to the $p$-adic setting.

We tried two other (more natural) approaches for proving \Cref{circuits-DGSC}, each of which interestingly failed. Let $G=(V,E)$ be a graph, and $w\in \mathbb{Z}^E_+$. It is known that $w\in \cone(\mathbf{C}(G))$ if, and only if, $w(\delta(U)\setminus f)\geq w_f$ for all cuts $\delta(U)$ and all $f\in \delta(U)$~\cite{Seymour79-b}. Suppose $w\in \cone(\mathbf{C}(G))$. If $w=\1$, then $G$ is bridgeless, so $2\cdot \1\in \mathbb{R}^E$ belongs to the lattice generated by $\mathbf{C}(G)$~\cite{Goddyn93}, so by using the Density Lemma one can show that $w$ can be written as a dyadic conic combination of the vectors in $\cone(\mathbf{C}(G))$, as desired. Otherwise, $w\neq \1$. One attempt would be to reduce to the all-ones case by replacing every edge $e$ by $w_e$ parallel edges, but this approach fails for the simple reason that the new graph has circuits (of length two) that do not correspond to circuits of $G$. Another attempt would be to reconstruct the proof of the $w=\1$ case for general $w$. This approach goes through if $w$ belongs to the relative interior of $\cone(\mathbf{C}(G))$, but if $w(\delta(v)\setminus f)-w_f=0$ for some $v\in V$ and $f\in \delta(v)$, then the approach fails.

A very interesting, relevant result is that a graph has the so-called \emph{circuit cover property} if, and only if, it has no Petersen minor~\cite{Alspach94}. Subsequently, if $G$ has no Petersen minor, then every vector in $\cone(\mathbf{C}(G))\cap \left\{x\in \mathbb{Z}^E: x(\delta(U)) \text{ is even } \forall U\subseteq V\right\}$ can be expressed as an integral linear combination of the vectors in $\mathbf{C}(G)$, so every integral vector in $\cone(\mathbf{C}(G))$ can be expressed as a half-integral linear combination of the vectors in $\mathbf{C}(G)$.

\section*{Acknowledgement}

We would like to thank the referees whose comments on an earlier draft improved the final presentation. Bertrand Guenin was supported by NSERC grant 238811. Levent Tun\c{c}el was supported by Discovery Grants from NSERC and ONR grant N00014-18-1-2078.

{\small \bibliographystyle{abbrv}\bibliography{references}}

\end{document}